\documentclass[reqno]{amsart}
\allowdisplaybreaks[4]

\usepackage{amsmath, amssymb}
\usepackage{amsthm}
\usepackage{enumitem}
\usepackage{mathrsfs}

\usepackage{hyperref}
\hypersetup{colorlinks=true,urlcolor=blue,linkcolor=blue,citecolor=blue}

\usepackage{url}

\setlist[enumerate]{label={\upshape (\roman*)}}

\makeatletter
	\@addtoreset{equation}{section}

\theoremstyle{plain}
\newtheorem{theorem}{\indent\sc Theorem}[section]
\newtheorem{lemma}[theorem]{\indent\sc Lemma}
\newtheorem{corollary}[theorem]{\indent\sc Corollary}
\newtheorem{proposition}[theorem]{\indent\sc Proposition}

\theoremstyle{definition}
\newtheorem{definition}[theorem]{\indent\sc Definition}
\newtheorem{remark}[theorem]{\indent\sc Remark}

\newcommand{\cE}{\mathcal{E}}
\newcommand{\cF}{\mathcal{F}}

\begin{document}

\title{Dynkin games for Markov processes associated with semi-Dirichlet forms}
\author[T. Ooi \and T. Uemura]{Takumu Ooi \and Toshihiro Uemura}
\address[T.Ooi]{Research Institute for Mathematical Sciences, Kyoto University, Kyoto-shi, Kyoto 606-8502, Japan.}
\email[T.Ooi]{ooitaku@kurims.kyoto-u.ac.jp}
\address[T.Uemura]{Department of Mathematics, Faculty of Engineering Science, Kansai University, Suita-shi, Osaka 564-8680, Japan.}
\email[T.Uemura]{t-uemura@kansai-u.ac.jp} 

\begin{abstract}
We consider Dynkin games for Markov processes associated with semi-Dirichlet forms. Dynkin games are the optimal stopping games introduced as the models of zero-sum games by two players. We prove that the solution to the certain variational inequality with two obstacles is the equilibrium price of the Dynkin game. Moreover, we obtain the saddle point of the game.
\end{abstract}

\maketitle

\section{Introduction}\label{Introduction}
We consider the optimal stopping game for Markov process, which is originally considered by Dynkin \cite{D} as the model of zero-sum game by two players, and this is called the Dynkin game. Nowadays the Dynkin game is known as a generalization of an optimal stopping problem and has been studied under various settings (see \cite{K, SY} for the history, for example). In \cite{N}, Nagai revealed a relation between the solutions to the variational inequality with one obstacle for a regular Dirichlet form and to the optimal stopping problem for the corresponding symmetric Markov process. Zabczyk \cite{Z} proved the existence of solution to the variational inequality with two obstacles for a regular Dirichlet form and also showed that the solution is the equilibrium price of the value function of the Dynkin game for the corresponding symmetric Markov process. Zhang \cite{Zh} generalized
Nagai's result to Markov processes associated with semi-Dirichlet forms. In this paper, we consider a Dynkin game for Markov processes associated with semi-Dirichlet forms.

To explain the Dynkin games for Markov processes associated with semi-Dirichlet forms, we recall definitions and some properties for semi-Dirichlet forms from \cite{O}. Let \(E\) be a locally compact separable metric space, \(\mathcal{B}(E)\) be the Borel \(\sigma\)-field on \(E\), and \(m\) be an everywhere dense positive Radon measure on \(E.\) Denote by \((\cdot, \cdot)\) and \(\|\cdot\|\) the inner product and the norm in \(L^2(E;m)\), respectively. If a dense linear subspace \(\cF\) and a bilinear form \(\cE\) defined on \(\cF \times \cF\) satisfy the following conditions \((\cE.1)\), \((\cE.2)\), \((\cE.3)\) and \((\cE.4)\), then \((\cE, \cF)\) is called a lower-bounded semi-Dirichlet form on \(L^2(E;m)\).

\((\cE.1)\) The lower boundedness : There exists \(\alpha_0\geq 0\) such that \(\cE_{\alpha_0}(u,u)\geq 0\) for any \(u\in \cF,\) where \(\cE_{\alpha}(u,v)=\cE(u,v)+\alpha(u,v)\) for \(u,v\in \cF\) and \(\alpha \geq 0.\)

\((\cE.2)\) The sector condition : There exists \(K\geq 1\) such that 
\[|\cE(u,v)|\leq K
\sqrt{\cE_{\alpha_0}(u,u)} \sqrt{\cE_{\alpha_0}(v,v)} \text{\ for\ any\ } u,v\in \cF.\]

\((\cE.3)\) The closedness : \(\cF\) is a Hilbert space relative to the inner product \[\cE_{\alpha}^{(s)}(u,v)=\frac{1}{2}(\cE_{\alpha}(u,u)+\cE_{\alpha}(v,u)) \text{\ for\ any\ }\alpha>\alpha_0.\]

\((\cE.4)\) The Markov property : For any \(u\in \cF\) and \(a\geq 0,\) it holds that \(u\wedge a \in \cF\) and \(\cE(u\wedge a, u-u\wedge a)\geq 0.\)

By \((\cE.2)\), for any \(\alpha>\alpha_0,\) there exists a constant \(K_{\alpha}>0\)  depending on \(\alpha\) such that 
\begin{equation}
|\cE_{\alpha}(u,v)|\leq K_{\alpha}\sqrt{\cE_{\alpha}(u,u)} \sqrt{\cE_{\alpha}(v,v)} \text{\ for\ any\ } u,v\in \cF.
\label{eq:sector}\end{equation}
A function \(u\in \cF\) is called \(\alpha\)-potential if \(\cE_{\alpha}(u,v)=\int v d\mu\) for any \(v\in C_c(E)\cap \cF\) for some positive Radon measure \(\mu.\) Here and throughout this paper, \(C_c(E)\) denotes the family of all continuous functions with compact supports. A lower-bounded semi-Dirichlet form \((\cE, \cF)\) on \(L^2(E;m)\) is regular if \(C_c(E)\cap \cF\) is \(\cE_{\alpha_0}\)-dense in \(\cF\) and uniformly dense in \(C_c(E)\). For a regular lower-bounded semi-Dirichlet form \((\cE, \cF)\) on \(L^2(E;m)\), by \cite[Theorem 3.3.4]{O}, there exists a Hunt process \(X=(\{X_t\}_{t\geq 0}, \{\mathbb{P}_x\}_{x\in E})\) on \(E\) properly associated with \((\cE, \cF)\). Here, a Hunt process is a strong Markov process possessing quasi-left continuity. We remark that this process \(X\) is not always \(m\)-symmetric. See \cite{O} for details.

Moreover, we define quasi notations. For an open set \(A\), if \(\mathscr{L}_A^{\alpha}:=\{u\in \cF : u\geq 1 \text{\  on\  }A,\ m\text{-a.e.}\}\) is no-empty, there exists \(e_A^{\alpha} \in \cF\) (resp. \(\hat{e}_A^{\alpha}\in \cF\)) such that \(\cE_{\alpha}(e_A^{\alpha}, u-e_A^{\alpha})\geq 0\) (resp. \(\cE_{\alpha}(u-\hat{e}_A^{\alpha}, \hat{e}_A^{\alpha})\geq 0\)) for \(u\in \mathscr{L}_A^{\alpha}\) by \cite[p. 43]{O}. In this paper, we use \(:=\) as a way of definition. Define the \(\alpha\)-capacity \(\text{Cap}^{(\alpha)}(A)\) of an open set \(A\) by 
\begin{equation*}
\text{Cap}^{(\alpha)}(A):= \left \{ \begin{split} \cE_{\alpha}(e_A^{\alpha}, \hat{e}_A^{\alpha}) &\text{\ if\ }&\mathscr{L}_A^{\alpha}\neq \emptyset,\\
\infty  \hspace{5mm}&\text{\ if\ }&\mathscr{L}_A^{\alpha}= \emptyset,
\end{split} \right.
\end{equation*}
and, define \(\text{Cap}^{(\alpha)}(B):=\inf\{\text{Cap}^{(\alpha)}(A) : A \text{\ is\ open,\ }B \subset A\}\) for any \(B\subset E\). Using this capacity, we say the statement holds quasi-everywhere (q.e.) on \(A\) if this statement holds for any \(x\in A \setminus N\) for some \(N\) with \(\text{Cap}^{(\alpha)}(N)=0.\) Moreover, we say the function \(u\) is quasi-continuous if, for any \(\varepsilon >0,\) there exists an open set \(A\) such that \(\text{Cap}^{(\alpha)}(A)<\varepsilon\) and the restriction of \(u\) on \(A^c\) is continuous. For hitting times \(\sigma_N:=\inf\{t>0:X_t\in N\}\) and \(\hat{\sigma}_N:=\inf\{t>0:X_{t-}\in N\}\), \(N\) is properly exceptional if \(N\) is an \(m\)-negligible set and \(\mathbb{P}_x(\sigma_N \wedge \hat{\sigma}_N <\infty)=0\) for any \(x\in N^c.\) Here and throughout this paper, for real numbers \(a\) and \(b,\) we use \(a \wedge b\) and \(a \vee b\) as the meaning of the maximum and the minimum of \(a\) and \(b\), respectively.

We state the main results. We assume \((\cE, \cF)\) is a regular lower bounded semi-Dirichlet form on \(L^2(E)\) and let \(X=(\{X_t\}_t, \{\mathbb{P}_x\}_{x\in E})\) be the Markov process associated with \((\cE, \cF)\). The following theorem is proved by Nagai \cite{N} in the case that \((\cE,\cF)\) is a symmetric Dirichlet form, and essentially by Zhang \cite{Zh} in the case that \((\cE,\cF)\) is a semi-Dirichlet form. 
\begin{theorem}\label{Zhang}
For any \(g\in \cF,\) there exists \(w\in \cF\) satisfying \(g\leq w\ m\)-a.e. such that the following variational inequality holds;
\begin{equation}
\cE_{\alpha}(w,u-w)\geq 0\ \text{for\ any\ }u\in \cF\ \text{satisfying}\  g\leq u\ m\text{-a.e.}
\label{eq:Zhangvarineq}\end{equation}
and, for the quasi-continuous version \(\tilde{w}\) of \(w\), there exists a properly exceptional set \(N\) such that
\begin{equation}
\tilde{w}(x)=\sup_{\sigma}\mathbb{E}_x[e^{-\alpha \sigma}\tilde{g}(X_{\sigma})]=\mathbb{E}_x[e^{-\alpha \hat{\sigma}}\tilde{g}(X_{\hat{\sigma}})],\ \text{for\ any\ }x\in N^c,
\label{eq:Zhangsaddle}\end{equation}
where \(\tilde{g}\) is the quasi-continuous version of \(g\), the supremum is taken over the set of all stopping times for \(X,\) and \(\hat{\sigma}:=\inf \{t\geq 0; \tilde{w}(X_t)=\tilde{g}(X_t)\}\). Moreover, \(w\) is the smallest \(\alpha\)-potential majorizing \(g\ m\)-a.e. in the sense that \(w\leq \check{w}\) \(m\)-a.e. for any \(\alpha\)-potential \(\check{w} \in \cF\) with \(g\leq \check{w}\ m\)-a.e.
\end{theorem}

For any stopping times \(\tau\) and \(\sigma\) for \(X\), we define
\[J_x(\tau,\sigma):=\mathbb{E}_x[e^{-\alpha(\tau \wedge \sigma)}({\bf 1}_{\{\tau \leq \sigma\}}\tilde{h}(X_t)+{\bf 1}_{\{\tau > \sigma\}}\tilde{g}(X_t))],\]
where \(\tilde{g}\) and \(\tilde{h}\) are the quasi-continuous versions of \(g\) and \(h\), respectively. We remark that every function in \(\cF\) has the quasi-continuous version.

Then we state the main theorem concerning the variational inequality with two obstacles for a regular lower-bounded semi-Dirichlet form \((\cE, \cF)\) and the Dynkin game for the Markov process associated with \((\cE, \cF)\).

\begin{theorem}\label{mainThm}
There exists \(v\in \cF\) satisfying \(g\leq v\leq h\ m\)-a.e. such that the following variational inequalities holds;

\begin{equation}
\cE_{\alpha}(v,u-v)\geq 0\ \text{for\ any\ }u\in \cF\ \text{satisfying}\  g\leq u\leq h\ m\text{-a.e.}
\label{eq:varineq}\end{equation}
and, for the quasi-continuous version \(\tilde{v}\) of \(v\), there exists a properly exceptional set \(N\) such that
\begin{equation}
\tilde{v}(x)=\sup_{\sigma}\inf_{\tau}J_x(\tau, \sigma) = \inf_{\tau}\sup_{\sigma}J_x(\tau, \sigma) \text{\ for\ }x\in N^c,
\label{eq:supinf}\end{equation}
where \(\tilde{v}\) is the quasi-continuous version of \(v\), the supremum and the infimum are taken over the set of all stopping times for \(X.\)

Moreover, stopping times \(\hat{\tau}:=\inf \{t\geq 0; \tilde{v}(X_t)=\tilde{h}(X_t)\}\) and \(\hat{\sigma}:=\inf \{t\geq 0; \tilde{v}(X_t)=\tilde{g}(X_t)\}\) are the saddle point of the game in the sense that
\begin{equation}
\tilde{v}(x)=J_x(\hat{\tau}, \hat{\sigma})\ \text{and\ }J_x(\hat{\tau}, \sigma)\leq J_x(\hat{\tau}, \hat{\sigma}) \leq J_x(\tau, \hat{\sigma})
\label{eq:saddle}\end{equation}
hold for any stopping times \(\tau\) and \(\sigma\), and \(x\in N^c.\)
\end{theorem}

The above theorem is a generalization of \cite{Z} to the case of semi-Dirichlet forms. We prove Theorem \ref{mainThm} in a similar way to those of \cite{Z} and \cite{FM}. The differences from the proof of \cite{Z, FM} are the convergences in \(\cE_{\alpha}\) of \(\alpha\)-potentials. In the symmetric case, bounded increasing \(\alpha\)-potentials converge to some \(\alpha\)-potential since an \(\alpha\)-potential is also an \(\alpha\)-copotential. However, in the case of a semi-Dirichlet form, an \(\alpha\)-potential is not always an \(\alpha\)-copotential, so we cannot use the exactly same proof as that of \cite{Z,FM}. Instead of these convergences, by applying Banach-Saks's theorem, we use the convergence of the Ces\`{a}ro mean of a subsequence.

The rest of this paper is organized as follows. In section \ref{Proof of Theorem 1.1}, we prove that Theorem \ref{Zhang} follows essentially from \cite[Theprem 3]{Zh}. In section \ref{Proof of Theorem 1.2}, we prove Theorem \ref{mainThm} in a similar way to \cite{Z, FM} except for the convergences in \(\cE_{\alpha}\) for \(\alpha\)-potentials since this convergence follows from the symmetry of the process. Instead of these convergences, we mainly use Banach Saks's theorem. Section \ref{examples} presents examples.

\section{Proof of Theorem 1.1}\label{Proof of Theorem 1.1}
Let \((\cE,\cF)\) be a regular lower-bounded semi-Dirichlet forms on \(L^2(E;m)\) and \(X\) be a Hunt process properly associated with \((\cE,\cF)\). Fix \(\alpha >\alpha_0.\) In this section, we prove that Theorem \ref{Zhang} essentially follows from \cite[Theorem 3]{Zh}, which is the relation between the variational inequality with one obstacle for a non-negative definite semi-Dirichlet form and the optimal stopping game for the associated Markov process. Since every function in \(\cF\) has the quasi-continuous version, we assume that all elements in \(\cF\) are quasi-continuous without loss of generality.

For reader's convenience, we describe \cite[Theorem 3]{Zh}.
\begin{theorem}{\((\)\cite[Theorem 3]{Zh}\()\)}\label{OriginalZhang}
For any \(G\in \cF\) and \(F\in L^2(E;m)\), we define
\[\Phi(x):=\inf_{\tau} \mathbb{E}_x\left[e^{-\alpha \tau} G(X_{\tau})+\int_0^{\tau}e^{-\alpha s}F(X_s)ds\right].\]
Then, for the stopping time \(\tau^*:=\inf\{t\geq 0 ; \Phi(X_t)\geq G(X_t)\},\) it holds that
\[\Phi(x)=\mathbb{E}_x\left[e^{-\alpha \tau^*} G(X_{\tau^*})+\int_0^{\tau^*}e^{-\alpha s}F(X_s)ds\right]\] for q.e. \(x\), and \(\Phi \) is the maximum solution to the following inequalities ;
\[\cE_{\alpha}(\Phi,\phi)\leq (F, \phi)\ \text{for\ }\phi\in \cF\ \text{with\ }\phi\geq 0 \text{ and\ }\Phi \leq G\ m\text{-a.e}.\]
\end{theorem}

The following proposition follows from Theorem \ref{OriginalZhang} immediately.
\begin{proposition}
For any \(g\in \cF\) and \(f\in L^2(E;m)\), there exists \(w\in \cF\) satisfying \(w\geq g\) \(m\)-a.e. such that, \[\cE_{\alpha}(w,u-w)\geq (f,u-w)\ \text{for\ }u\in \cF\ \text{with\ }u\geq g \ m\text{-a.e}.\]
Moreover, there exists a properly exceptional set \(N\) such that
\begin{eqnarray}
\nonumber w(x)&=&\sup_{\tau} \mathbb{E}_x\left[e^{-\alpha \tau} g(X_{\tau})+\int_0^{\tau}e^{-\alpha s}f(X_s)ds\right]\\
&=&\mathbb{E}_x\left[e^{-\alpha \hat{\tau}} g(X_{\hat{\tau}})+\int_0^{\hat{\tau}}e^{-\alpha s}f(X_s)ds\right] \label{eq:OZhang}
\end{eqnarray}
for \(x\in N^c.\)
\label{Zhangprop}\end{proposition}

\begin{proof}
For \(G:=-g\) and \(F:=-f\), let \(\Phi\) be the maximum solution of \((\ref{eq:OZhang})\). Then \(w=-\Phi\) is the solution. By the proof of \cite[Theorem 3]{Zh}, there exist 
\(\Phi_{\varepsilon}\in \cF\) such that \(\Phi_{\varepsilon}\searrow \Phi\) \(m\)-a.e. and  in \(\cE_{\alpha}\). Moreover, \(\cE_{\alpha}(\Phi_{\varepsilon}, \phi)=(F-\frac{1}{\varepsilon}(\Phi_{\varepsilon}-G)\vee 0, \phi)\) holds for any \(\phi \in \cF\). So, for \(u\in \cF\) with \(u\geq -G\), we have 
\[\cE_{\alpha}(\Phi_{\varepsilon}, u+\Phi_{\varepsilon})=(F,u+\Phi_{\varepsilon})-\frac{1}{\varepsilon}((\Phi_{\varepsilon}-G)\vee 0, u+\Phi_{\varepsilon}).\]
If \(\Phi_{\varepsilon}\geq G\), it holds that \(u+\Phi_{\varepsilon}\geq -G+\Phi_{\varepsilon}\geq 0\), so we have \(\cE_{\alpha}(\Phi_{\varepsilon}, u+\Phi_{\varepsilon})\leq (F,u+\Phi_{\varepsilon}),\) and by letting \(\varepsilon \) to \(0\), it holds that \(\cE_{\alpha}(\Phi, u+\Phi)\leq (F,u+\Phi)\) for any \(u\geq -G.\) By considering \(g=-G, w=-\Phi\) and \(f=-F\), the proof is completed.
\end{proof}

To prove that the solution \(w\) is an \(\alpha\)-potential, we need the following lemma.

\begin{lemma}
Under the assumptions of Proposition \ref{Zhangprop},\\
\((1)\) if \(f\geq 0\), then the solution \(w\) of Proposition \ref{Zhangprop} is an \(\alpha\)-potential.\\
\((2)\) if \(f=0,\) then \(w\) is the smallest \(\alpha\)-potential majorizing \(g\ m\)-a.e.
\label{Zhanglem}\end{lemma}
\begin{proof}
\((1)\) Suppose \(f\geq 0\). By the proof of Proposition \ref{Zhangprop} and Theorem \ref{OriginalZhang}, it hold that
\(\cE_{\alpha}(w,\phi)\geq (f,\phi)\geq 0\) for any \(\phi \in \cF\) with \(\phi\geq 0\). So \(w\) is an \(\alpha\)-potential by \cite[Theorem 1.4.1, Theorem 2.3.1]{O}.

\((2)\) Suppose \(f=0\). Then \(w\) is an \(\alpha\)-potential by \((1)\) and, for an \(\alpha\)-potential \(\check{w}\in \cF\) with \(g\leq \check{w}\ m\)-a.e., \(w \wedge \check{w}\) is also an \(\alpha\)-potential by \cite[Corollary 2.3.2]{O}. So we have \(\cE_{\alpha}(w \wedge \check{w},w \wedge \check{w}-w)\leq 0.\) Moreover, by Proposition \ref{Zhangprop}, we have
\[\cE_{\alpha}(w,w \wedge \check{w}-w)\geq (f,w \wedge \check{w}-w)= 0.\]
Thus we have \(\cE_{\alpha}(w-w\wedge \check{w}, w-w\wedge \check{w})=0\) and so \(w \wedge \check{w}=w\) holds \(m\)-a.e.
\end{proof}

\begin{proof}[Proof of Theorem \ref{Zhang}]
This follows from Proposition \ref{Zhangprop} and Lemma \ref{Zhanglem}.
\end{proof}

\section{Proof of Theorem 1.2}\label{Proof of Theorem 1.2}
In this section, we prove Theorem \ref{mainThm}. We prove \((\ref{eq:varineq})\) and \((\ref{eq:supinf})\) in a similar way to the proof of \cite{Z}. The difference is the use of the convergence of the Ces\`{a}ro mean of subsequence instead of the convergence of the approximate sequence of the solution to the variational inequality. This is because of the lack of the symmetry of \((\cF,\cE)\). So we need to check that the Ces\`{a}ro mean of subsequence also satisfy the appropriate variational inequalities. 

We prove Theorem \ref{mainThm} under the separability condition, and we use the approximation of the general case by the cases enjoying the separability conditions. As in the previous section, let \((\cE,\cF)\) be a regular lower-bounded semi-Dirichlet form on \(L^2(E;m)\) and \(X\) be the Hunt process associated with \((\cE,\cF)\). Without loss of generality, we may assume all functions in \(\cF\) are quasi-continuous.

\subsection{The separability condition}
For \(g,h\in \cF\) with \(g\leq h\) \(m\)-a.e., we recall the definition of the separability condition from \cite{Z}.
\begin{definition}
For \(g,h\in \cF\) with \(g\leq h\) \(m\)-a.e., the pair \((g, h)\in \cF \times \cF\) satisfies the separability condition if there exist \(\alpha\)-potentials \(w_1, w_2 \in \cF\) such that \(g\leq w_1-w_2 \leq h\) \(m\)-a.e.
\label{separability}\end{definition}
First, we prove Theorem \ref{mainThm} under the separability condition.

\begin{proposition}\label{prop1}
The separability condition holds if and only if there exist \(\alpha\)-potentials \(\overline{v}, \underline{v}\in\cF\) such that 
\begin{eqnarray}
\overline{v}\geq \underline{v}+g,\ \cE_{\alpha}(\overline{v},u-\overline{v})\geq 0\ \text{for\ }u\in \cF \ \text{with\ }u\geq \underline{v}+g,\label{uebar}
\end{eqnarray}
and
\begin{eqnarray}
\underline{v}\geq \overline{v}-h,\ \cE_{\alpha}(\underline{v},u-\underline{v})\geq 0\ \text{for\ }u\in \cF \ \text{with\ }u\geq \overline{v}-h.
\label{sitabar}\end{eqnarray}
\label{sepequi}\end{proposition}

\begin{remark}
Proposition \ref{sepequi} is proved by Zabczyk \cite[Proposition 1]{Z} in the case that \((\cE, \cF)\) is a Dirichlet form. In our case, we prove Proposition \ref{sepequi} in a similar way to the proof of \cite[proposition 1]{Z} except for the convergence of \(\overline{v}_n \to \overline{v}\) and \(\underline{v}_n \to \underline{v}\) in \(\cE_{\alpha}.\)
\end{remark}

\begin{proof}[Proof of Proposition \ref{prop1}]
Suppose that there exist \(\alpha\)-potentials \(\overline{v}, \underline{v}\in\cF\) such that \((\ref{uebar})\) and \((\ref{sitabar})\) hold. Then the separability condition holds by taking \(w_1=\overline{v}\) and \(w_2=\underline{v}\).

Suppose that the separability condition holds. We fix \(\alpha\)-potentials \(w_1,w_2 \in \cF\) satisfying \(g\leq w_1-w_2 \leq h\) \(m\)-a.e. We define sequences \(\{\overline{v}_n\}_n\) and \(\{\underline{v}_n\}_n\) as follows inductively. Let \(\overline{v}_0=\underline{v}_0:=0\). For \(\overline{v}_{n}\) and \(\underline{v}_{n}\), by using Proposition \ref{Zhang}, we define \(\alpha\)-potentials \(\overline{v}_{n+1}, \underline{v}_{n+1}\in \cF\) satisfying \(\underline{v}_{n}+g\leq \overline{v}_{n+1}\) \(m\)-a.e., \(\overline{v}_{n}-h\leq \underline{v}_{n+1}\) \(m\)-a.e.,
\begin{eqnarray}
\cE_{\alpha}(\overline{v}_{n+1},u-\overline{v}_{n+1})\geq 0\ \text{for\ any\ }u\in \cF\ \text{satisfying}\  \underline{v}_{n}+g\leq u\ m\text{-a.e.} \label{eq:uesep}
\end{eqnarray}
and
\begin{eqnarray}
\cE_{\alpha}(\underline{v}_{n+1},u-\underline{v}_{n+1})\geq 0\ \text{for\ any\ }u\in \cF\ \text{satisfying}\  \overline{v}_{n}-h\leq u\ m\text{-a.e.} \label{eq:sitasep}
\end{eqnarray}
Moreover we find \(\overline{v}_1\leq \overline{v}_2 \leq \cdots \leq \overline{v}_n \leq \cdots \leq w_1 \) and \(\underline{v}_1\leq \underline{v}_2 \leq \cdots \leq \underline{v}_n \leq \cdots \leq w_2 \) \(m\)-a.e. Indeed, we can prove these as follows inductively. Both \(\overline{v}_1\) and \(\underline{v}_1\) are non-negative since these are \(\alpha\)-potentials. If \(\underline{v}_n\leq \underline{v}_{n+1}\), then we have \(\underline{v}_n+g \leq \underline{v}_{n+1}+g \leq \overline{v}_{n+2}\). Since \(\overline{v}_{n+1}\) is the smallest \(\alpha\)-potential majorizing \(\underline{v}_n+ g\ m\)-a.e., \(\overline{v}_{n+1}\leq \overline{v}_{n+2}\) holds. The other can be proved similarly.

Since \(\overline{v}_n\) is an \(\alpha\)-potential, by \((\ref{eq:uesep})\) and \((\ref{eq:sector})\), we have
\[\cE_{\alpha}(\overline{v}_n,\overline{v}_n)\leq \cE_{\alpha}(\overline{v}_n,w_1)\leq K_{\alpha}\cE_{\alpha}(\overline{v}_n,\overline{v}_n)^{1/2}\cE_{\alpha}(w_1,w_1)^{1/2}.\]
So \(\{\cE_{\alpha}(\overline{v}_n,\overline{v}_n)\}_n\) is bounded. Similarly, \(\{\cE_{\alpha}(\underline{v}_n,\underline{v}_n)\}_n\) is also bounded. By Banach Saks's theorem, there exists subsequence \(\{n_k\}_k\) and \(\overline{v}, \underline{v}\in \cF\) such that \(\frac{1}{N}\sum_{k=1}^N\overline{v}_{n_k}\) converges to \(\overline{v}\) in \(\cE_{\alpha}\) and \(\frac{1}{N}\sum_{k=1}^N\underline{v}_{n_k}\) converges to \(\underline{v}\) in \(\cE_{\alpha}\) as \(N\to \infty\). For any \(u\in \cF\) with \(u\geq \underline{v}+g\), we have
\begin{eqnarray*}
&&\cE_{\alpha}(\frac{1}{N}\sum_{k=1}^N\overline{v}_{n_k},u-\frac{1}{N}\sum_{k=1}^N\overline{v}_{n_k})\\
&&\hspace{10mm} = \frac{1}{N^2} \sum_{k=1}^N\sum_{j=1}^N \cE_{\alpha}(\overline{v}_{n_k},u-\overline{v}_{n_j})\\
&&\hspace{10mm}= \frac{1}{N^2} \sum_{k=1}^N\sum_{j=1}^N \left\{ \cE_{\alpha}(\overline{v}_{n_k},u-\overline{v}_{n_k}) + \cE_{\alpha}(\overline{v}_{n_k},\overline{v}_{n_k}-\overline{v}_{n_j}) \right\} \\
\end{eqnarray*}

Since \(\{\underline{v}_{n_k}\}_{n_k}\) is increasing, so is \(\{\frac{1}{N}\sum_{k=1}^N\underline{v}_{n_k}\}_N\). By taking further subsequence and quasi-continuous versions if necessary, \(\frac{1}{N}\sum_{k=1}^N\underline{v}_{n_k}\) converges to \(\underline{v}\) q.e. by \cite[Theorem 2.2.5]{O}. Then we have 
\[\frac{1}{N}\sum_{k=1}^N\underline{v}_{n_k} \geq \frac{1}{N}\sum_{k=1}^m\underline{v}_{n_k}+ \frac{N-m}{N}\underline{v}_{n_m}\]
for fixed \(m\), and, by letting \(N\) tend to \(\infty\), \(\underline{v} \geq \underline{v}_{n_m}\) holds. So we have \(u\geq \underline{v}+g \geq \underline{v}_{n_k}+g \) q.e. Combining this with \((\ref{eq:uesep})\), it holds that
\begin{eqnarray*}
&&\frac{1}{N^2} \sum_{k=1}^N\sum_{j=1}^N \left\{ \cE_{\alpha}(\overline{v}_{n_k},u-\overline{v}_{n_k}) + \cE_{\alpha}(\overline{v}_{n_k},\overline{v}_{n_k}-\overline{v}_{n_j}) \right\} \\
&&\hspace{10mm}\geq  \frac{1}{N^2} \sum_{k=1}^N\sum_{j=1}^N \cE_{\alpha}(\overline{v}_{n_k},\overline{v}_{n_k}-\overline{v}_{n_j})\\
&&\hspace{10mm} = \frac{1}{N^2} \sum_{k<j}\cE_{\alpha}(\overline{v}_{n_k},\overline{v}_{n_k}-\overline{v}_{n_j})+\frac{1}{N^2} \sum_{k>j}\cE_{\alpha}(\overline{v}_{n_k},\overline{v}_{n_k}-\overline{v}_{n_j})\\
&&\hspace{10mm} =\frac{1}{N^2} \sum_{k<j}\cE_{\alpha}(\overline{v}_{n_k},\overline{v}_{n_k}-\overline{v}_{n_j})+\frac{1}{N^2} \sum_{k<j}\cE_{\alpha}(\overline{v}_{n_j},\overline{v}_{n_j}-\overline{v}_{n_k})\\
&&\hspace{10mm} =\frac{1}{N^2} \sum_{k<j}\cE_{\alpha}(\overline{v}_{n_k}-\overline{v}_{n_j},\overline{v}_{n_k}-\overline{v}_{n_j})\\
&&\hspace{10mm}\geq  0.
\end{eqnarray*}
In the second term of the second equality, we replaced \(j\) with \(k.\) By letting \(N\) tend to \(\infty\), we have \(\cE_{\alpha}(\overline{v}, u-\overline{v})\geq 0.\)

Similarly, \(\cE_{\alpha}(\underline{v}, u-\underline{v})\geq 0\) for any \(u\in \cF\) with \(u\geq \overline{v}-h\) \(m\)-a.e.

Moreover, since \(\overline{v}_{n_{k+1}}\geq \overline{v}_{n_k+1} \geq \underline{v}_{n_k}+g\), we have
\begin{eqnarray}
\nonumber \frac{1}{N}\sum_{k=1}^{N}\overline{v}_{n_{k}}&\geq & \frac{1}{N}\sum_{k=1}^{N} \underline{v}_{n_{k}}+g+ \frac{\underline{v}_{n_0}-\underline{v}_{n_N}}{N}\\
&\geq &\frac{1}{N}\sum_{k=1}^{N} \underline{v}_{n_{k}}+g+ \frac{\underline{v}_{n_0}-w_2}{N}.
\label{eq:cesaro}\end{eqnarray}
By letting \(N\) tend to \(\infty\), we have \(\overline{v}\geq \underline{v}+g\) \(m\)-a.e. Similarly, we have \(\underline{v}\geq \overline{v}-h\) \(m\)-a.e.

Furthermore, we have \(\cE_{\alpha}(\overline{v},\phi)=\cE_{\alpha}(\overline{v},\phi+\overline{v}-\overline{v})\geq 0\) for any \(\phi \geq 0\) by \((\ref{eq:uesep})\), and so \(\overline{v}\) is an \(\alpha\)-potential. Similarly \(\underline{v}\) is also an \(\alpha\)-potential. The prof is completed.
\end{proof}

\begin{corollary}\label{sepcor}
Suppose that the separability condition holds. For \(\overline{v}, \underline{v}\) which appeared in Proposition \ref{sepequi}, let \(v:=\overline{v}-\underline{v}\). Then \(v\) is the unique solution of \((\ref{eq:varineq})\) satisfying \(g\leq v\leq h\) \(m\)-a.e.
\end{corollary}

To prove \((\ref{eq:supinf})\) and \((\ref{eq:saddle})\) under the separability condition, we use the following lemma.
\begin{lemma}\label{supermartingale}
Under the setting of Corollary \ref{sepcor}, there exists a properly exceptional set \(N\) such that, for any \(x\in N^c\),
\begin{eqnarray}
\overline{v}(x)=\mathbb{E}_x[e^{-\alpha \sigma} \overline{v}(X_{\sigma})] \text{\ for\ any\ stopping\ time\ }\sigma \leq \hat{\sigma},\label{eq:sup1}\\
\underline{v}(x)=\mathbb{E}_x[e^{-\alpha \tau} \underline{v}(X_{\tau})] \text{\ for\ any\ stopping\ time\ }\tau \leq \hat{\tau}\label{eq:sup2}.
\end{eqnarray}
\end{lemma}
\begin{proof}
Although the proof is the same as that appearing in \cite[Proposition 2]{Z}, we describe the proof of this lemma.

Since \(v=\overline{v}-\underline{v}\), we have \(\hat{\sigma}=\inf \{t\geq 0; \overline{v}(X_t)=(\underline{v}+g)(X_t)\}\). By Theorem \ref{Zhang}, we have
\[\overline{v}(x)=\sup_{\sigma}\mathbb{E}_x[e^{-\alpha \sigma}(\underline{v}+g)(X_{\sigma})]=\mathbb{E}_x[e^{-\alpha \hat{\sigma}}\overline{v}(X_{\hat{\sigma}})]\ \text{for\ q.e.\ }x\]
On the other hand, since \(\overline{v}\) is an \(\alpha\)-potential and \(X\) enjoys the strong Markov property, \((\{e^{-\alpha t}\overline{v}(X_t)\}_t, \mathbb{P}_x)\) is a non-negative supermartingale for q.e. \(x\). By the optional sampling theorem, we have 
\[\overline{v}(x)=\mathbb{E}_x[e^{-\alpha \hat{\sigma}}\overline{v}(X_{\hat{\sigma}})] \leq \mathbb{E}_x[e^{-\alpha \sigma} \overline{v}(X_{\sigma})]\]
for q.e. \(x\) and any stopping time \(\sigma \leq \hat{\sigma}\). By using optional sampling theorem again, we have \(\mathbb{E}_x[e^{-\alpha \sigma} \overline{v}(X_{\sigma})] \leq \overline{v}(x)\) for q.e. \(x\) and any stopping time \(\sigma \). So \((\ref{eq:sup1})\) holds. Similarly, \((\ref{eq:sup2})\) holds.  
\end{proof}

\begin{theorem}\label{sepmain}
Suppose that the separability condition holds. Then Theorem \ref{mainThm} holds.
\end{theorem}
\begin{proof}
For \(\overline{v}, \underline{v}\) which appeared in Proposition \ref{sepequi}, let \(v:=\overline{v}-\underline{v}\). By Corollary \ref{sepcor}, \(v\) enjoys \((\ref{eq:varineq})\), so it is enough to prove \((\ref{eq:supinf})\) and \((\ref{eq:saddle})\). 
Since \((\{e^{-\alpha t}\underline{v}(X_t)\}_t, \mathbb{P}_x)\) is a non-negative supermartingale, it holds that \(\mathbb{E}_x[e^{-\alpha \sigma} \underline{v}(X_{\sigma})] \leq \underline{v}(x)\) for any stopping time \(\sigma \). Combining this with Lemma \ref{supermartingale} and \(v\leq h\), for any stopping time \(\tau\), we have
\begin{eqnarray*}
v(x)&=&\overline{v}(x)-\underline{v}(x)\\
&\leq &\mathbb{E}_x[e^{-\alpha (\hat{\sigma}\wedge \tau)} \overline{v}(X_{\hat{\sigma}\wedge \tau})]-\mathbb{E}_x[e^{-\alpha (\hat{\sigma}\wedge \tau)} \underline{v}(X_{\hat{\sigma}\wedge \tau})]\\
&=& \mathbb{E}_x[e^{-\alpha (\hat{\sigma}\wedge \tau)} v(X_{\hat{\sigma}\wedge \tau})] \\
&\leq &J_x(\tau, \hat{\sigma})
\end{eqnarray*}
for q.e. \(x\). Similarly, \(\overline{v}(x)\geq J_x(\hat{\tau}, \sigma)\) holds for any stopping time \(\sigma\) and q.e. \(x.\) So \((\ref{eq:saddle})\) holds and \(v(x)=J_x(\hat{\tau}, \hat{\sigma}).\) By taking \(\sup_{\sigma}\inf_{\tau}\) and \(\inf_{\tau}\sup_{\sigma}\), \((\ref{eq:saddle})\) implies \((\ref{eq:supinf})\).
\end{proof}

\subsection{General cases}
Next we prove Theorem \ref{mainThm} without the separability condition. 
We fix \(g,h\in \cF\) with \(g\leq h\) \(m\)-a.e.

\begin{lemma}\label{approxLem}
There exist sequences \(\{g_n\}_n, \{h_n\}_n\subset \cF\) such that the pair \((g_n,h_n)\) enjoys the separability condition for each \(n\), and \(g_n\) (resp. \(h_n\)) converges to \(g\) (resp. \(h\)) in \(\cE_{\alpha}\). 
\end{lemma}
\begin{proof}
The proof is the same as that of \cite[Proposition 4]{Z}.
\end{proof}

\begin{proof}[Proof of Theorem \ref{mainThm}]
Fist, we prove \((\ref{eq:varineq})\) and \((\ref{eq:supinf})\) in Theorem \ref{mainThm} in a similar way to that of \cite{Z}, except for the convergence of an approximate sequence of the solution to the variational inequality. Instead of considering convergence in \(\cE_{\alpha}\) of the approximate sequence, we show the convergence of the Ces\`{a}ro mean of its subsequence.

By Lemma \ref{approxLem}, we take sequences \(\{g_n\}_n, \{h_n\}_n\subset \cF\) such that the pair \((g_n,h_n)\) enjoys the separability condition for each \(n\), and \(g_n\) (resp. \(h_n\)) converges to \(g\) (resp. \(h\)) in \(\cE_{\alpha}\). So, there exists the unique solution \(v_n\in \cF\) of Theorem \ref{mainThm} replaced \(g,h\) by \(g_n, h_n\), respectively. Set 
\[J^n_x(\tau, \sigma):=\mathbb{E}_x[e^{-\alpha (\tau \wedge \sigma)}({\bf 1}_{\{\tau \leq \sigma \}}h_n(X_{\tau})+{\bf 1}_{\{\tau>\sigma\}}g_n(X_{\sigma}))]\]
for stopping times \(\tau, \sigma.\) Then we have
\begin{eqnarray}
\nonumber |J^n_x(\tau, \sigma)-J_x(\tau, \sigma)| &\leq & \mathbb{E}_x[e^{-\alpha \tau}|h_n-h|(X_{\tau})]+\mathbb{E}_x[e^{-\alpha \sigma}|g_n-g|(X_{\sigma})]\\
&\leq & H_n(x)+G_n(x),\label{eq:HnGn}
\end{eqnarray}
where \(H_n(x):=\sup_{\tau}\mathbb{E}_x[e^{-\alpha \tau}|h_n-h|(X_{\tau})]\) and \(G_n(x):=\sup_{\sigma}\mathbb{E}_x[e^{-\alpha \sigma}|g_n-g|(X_{\sigma})].\)
Since \(H_n\) and \(G_n\) are the solutions to Theorem \ref{Zhang} replaced \(g\) by \(|h_n-h|\) and \(|g_n-g|\), respectively, we have
\begin{eqnarray*}
\cE_{\alpha}(H_n,H_n) &\leq &\cE_{\alpha}(H_n, |h_n-h|).
\end{eqnarray*}
By using (\ref{eq:sector}) and the Markov property for \((\cE,\cF)\), we have 
\begin{eqnarray*}
\cE_{\alpha}(H_n,H_n) &\leq &K_{\alpha}^2\cE_{\alpha}(|h_n-h|,|h_n-h|)\\
&\leq & K_{\alpha}^2\cE_{\alpha}(h_n-h,h_n-h).
\end{eqnarray*}
So \(H_n\) converges to \(0\) in \(\cE_{\alpha}\) as \(n \to \infty.\) Similarly, \(G_n\) also converges to \(0\) in \(\cE_{\alpha}\) as \(n \to \infty.\) By applying \cite[Theorem 2.2.5]{O} to any subsequence of \(\{H_n\}_n\) and \(\{G_n\}_n\), \(H_n\) and \(G_n\) converges to \(0\) q.e. as \(n \to \infty\). Since the pair \((g_n, h_n)\) enjoys the separability condition for each \(n\), there exists \(v_n\in \cF\) satisfying Theorem \ref{mainThm} replaced \(g,h\) by \(g_n, h_n\), respectively, and there exists properly exceptional set \(N_n\) such that \(v_n(x)=\inf_{\tau}\sup_{\sigma}J_x^n(\tau,\sigma)=\sup_{\sigma}\inf_{\tau}J_x^n(\tau,\sigma)\) for q.e. \(x\in N_n^c.\) Put \(u(x):=\inf_{\tau}\sup_{\sigma}J_x(\tau,\sigma)\) for q.e. \(x\in E\). Since both \(G_n\) and \(H_n\) do not depend on stopping times \(\tau, \sigma,\) we have \(|v_n-u|\leq |G_n+H_n|\) q.e. by \((\ref{eq:HnGn})\), and \(v_n\) converges to \(u\) q.e. as \(n\to \infty\). Similarly, \(v_n\) converges to \(\sup_{\sigma}\inf_{\tau}J_x(\tau,\sigma)\) q.e. and so \(u(x)=\inf_{\tau}\sup_{\sigma}J_x(\tau,\sigma)=\sup_{\sigma}\inf_{\tau}J_x(\tau,\sigma)\) q.e.

By \cite[Theorem 1.1.1]{O}, there is the solution \(v\in \cF\) to \((\ref{eq:varineq})\). To show that \(u=v\) q.e., we prove that \(u\) is also the solution to \((\ref{eq:varineq})\). Since \(u(x)=\inf_{\tau}\sup_{\sigma}J_x(\tau,\sigma)\) q.e., we have, for  q.e. \(x\),
\[u(x)\leq \sup_{\sigma}J_x(0,\sigma)= \sup_{\sigma}\mathbb{E}_xh(X_{0})=h(x)\] and
\[u(x)\geq \inf_{\tau}J_x(\tau,0)= \inf_{\tau} \mathbb{E}_xg(X_{0})=g(x).\]
Since \(v_n\) is the solution to \((\ref{eq:varineq})\) replaced \((g, h)\) by \((g_n, h_n)\), and \(g_n\to g\) in \(\cE_{\alpha}\), by \((\ref{eq:sector})\), we have
\begin{equation}
\cE_{\alpha}(v_n,v_n)\leq K_{\alpha}^2\cE_{\alpha}(g_n,g_n)\leq  K_{\alpha}^2(\cE_{\alpha}(g,g)+C)<\infty \label{eq:secv_n}
\end{equation}
for large \(n\) and some constant \(C.\) By Banach Saks's theorem, there exists a subsequence \(\{v_{n_k}\}_k\) such that its Ces\`{a}ro mean \(\frac{1}{N}\sum_{k=1}^Nv_{n_k}\) converges to some \(\check{v} \in \cF\) in \(\cE_{\alpha}\) as \(N\to \infty,\) and so \(\check{v}=u\) q.e.

For any \(\phi \in \cF\) with \(g\leq \phi \leq h\) \(m\)-a.e., we set \(\phi_n:=(g_n\vee \phi)\wedge h_n \in \cF.\) Then we have
\begin{eqnarray}
\nonumber &&\cE_{\alpha}(\frac{1}{N}\sum_{k=1}^Nv_{n_k}, \phi-\frac{1}{N}\sum_{k=1}^Nv_{n_k})\\
\nonumber &&\hspace{10mm} =\frac{1}{N^2}\sum_{k=1}^N\sum_{j=1}^N \cE_{\alpha}(v_{n_k}, \phi-v_{n_j})\\
\nonumber &&\hspace{10mm}= \frac{1}{N^2}\sum_{k=1}^N \sum_{j=1}^N \left\{\cE_{\alpha}(v_{n_k}, \phi-\phi_{n_k})+ \cE_{\alpha}(v_{n_k}, \phi_{n_k}-v_{n_k}) + \cE_{\alpha}(v_{n_k}, v_{n_k}-v_{n_j})\right\}\\
\nonumber &&\hspace{10mm}\geq  \frac{1}{N^2}\sum_{k=1}^N \sum_{j=1}^N \left\{\cE_{\alpha}(v_{n_k}, \phi-\phi_{n_k})+ 0 + \cE_{\alpha}(v_{n_k}, v_{n_k}-v_{n_j})\right\}\\
\nonumber &&\hspace{10mm}=\frac{1}{N}\sum_{k=1}^N \cE_{\alpha}(v_{n_k}, \phi-\phi_{n_k})+\frac{1}{N^2}\sum_{k<j} \cE_{\alpha}(v_{n_k}-v_{n_j}, v_{n_k}-v_{n_j})\\
&& \hspace{10mm}\geq \frac{1}{N}\sum_{j=1}^N \cE_{\alpha}(v_{n_k}, \phi-\phi_{n_k}). \label{eq:ces1}
\end{eqnarray}
Here we used \((\ref{eq:varineq})\) at the first inequality above. Since \(g_n\) converges to \(g\) and \(h_n\) converges to \(h\) in \(\cE_{\alpha}\), \(\phi_{n_k}\) converges to \(\phi\) in \(\cE_{\alpha}.\) Thus, by \((\ref{eq:secv_n})\), it holds that
\begin{eqnarray*}
|\cE_{\alpha}(v_{n_k}, \phi -\phi_{n_k})|& \leq &K_{\alpha}\cE_{\alpha}(v_{n_k}, v_{n_k})^{1/2}\cE_{\alpha}(\phi -\phi_{n_k}, \phi -\phi_{n_k})^{1/2}\\
&\leq & K_{\alpha}^2 (C+\cE_{\alpha}(g,g))^{1/2}\cE_{\alpha}(\phi -\phi_{n_k}, \phi -\phi_{n_k})^{1/2},
\end{eqnarray*}
so \(\cE_{\alpha}(v_{n_k}, \phi -\phi_{n_k})\) converges to \(0\) as \(n_k\to \infty\) and so does its Ces\`{a}ro mean \( \frac{1}{N}\sum_{j=1}^N \cE_{\alpha}(v_{n_k}, \phi -\phi_{n_k})\). Combining this with \((\ref{eq:ces1})\), by letting \(N\) tend to infinity, for any \(\phi\in \cF\) with \(g\leq \phi \leq h\) \(m\)-a.e, it holds that \(\cE_{\alpha}(u,\phi -u)\geq 0.\) This means that \(u\) is also a solution for \((\ref{eq:varineq})\) and so \(u=v\) q.e.

Next, we prove \((\ref{eq:saddle})\) in a similar way to that of \cite[Section 4]{FM} except for the convergence of approximate sequence \(\{v_n\}_n\) in \(\cE_{\alpha}.\) Instead of using approximate sequence \(\{v_n\}_n\), we take another approximation converging to the solution \(v.\)

Since the pair \((h_n, g_n)\) satisfies the separability condition, so does the pair \((\frac{1}{N}\sum_{k=1}^N h_{n_k}, \frac{1}{N}\sum_{k=1}^N g_{n_k})\). Thus there exists \(\check{v}_N \in \cF\) such that \(\frac{1}{N}\sum_{k=1}^N g_{n_k} \leq \check{v}_N \leq \frac{1}{N}\sum_{k=1}^N h_{n_k}\), \(m\)-a.e. and there exists properly exceptional set \(M_N\) such that 
\begin{eqnarray*}
\nonumber \check{v}_N(x)&=&\inf_{\tau}\sup_{\sigma} \mathbb{E}_x\left[e^{-\alpha (\tau \wedge \sigma)}(\frac{1}{N}\sum_{k=1}^N h_{n_k} (X_{\tau}) {\bf 1}_{\{\tau \leq \sigma\}}+\frac{1}{N}\sum_{k=1}^N g_{n_k} (X_{\sigma}) {\bf 1}_{\{\tau > \sigma\}})\right]\\
&=& \sup_{\sigma} \inf_{\tau}\mathbb{E}_x\left[e^{-\alpha (\tau \wedge \sigma)}(\frac{1}{N}\sum_{k=1}^N h_{n_k} (X_{\tau}) {\bf 1}_{\{\tau \leq \sigma\}}+\frac{1}{N}\sum_{k=1}^N g_{n_k} (X_{\sigma}) {\bf 1}_{\{\tau > \sigma\}})\right]
\end{eqnarray*}
for any \(x\in M_N^c.\) On the other hand, by \((\ref{eq:supinf})\), we have
\begin{eqnarray*}
v(x)&=&\inf_{\tau}\sup_{\sigma} \mathbb{E}_x\left[e^{-\alpha (\tau \wedge \sigma)}(h(X_{\tau}) {\bf 1}_{\{\tau \leq \sigma\}}+g(X_{\sigma}) {\bf 1}_{\{\tau > \sigma\}})\right]\\
&=&\sup_{\sigma}\inf_{\tau} \mathbb{E}_x\left[e^{-\alpha (\tau \wedge \sigma)}(h(X_{\tau}) {\bf 1}_{\{\tau \leq \sigma\}}+g(X_{\sigma}) {\bf 1}_{\{\tau > \sigma\}})\right]
\end{eqnarray*}
for \(x\in N^c.\) Let \(\check{h}_N:=\frac{1}{N}\sum_{k=1}^N h_{n_k}\) and \(\check{g}_N:=\frac{1}{N}\sum_{k=1}^N g_{n_k}\). Then we have
\begin{eqnarray}
\nonumber &&\left| \mathbb{E}_x[e^{-\alpha (\tau \wedge \sigma)}(\check{h}_N(X_{\tau}){\bf 1}_{\{\tau \leq \sigma\}} +\check{g}_N(X_{\sigma}){\bf 1}_{\tau >\sigma})]\right.\\
\nonumber && \hspace{20mm} \left.-\mathbb{E}_x[e^{-\alpha (\tau \wedge \sigma)}(h(X_{\tau}){\bf 1}_{\{\tau \leq \sigma\}} +g(X_{\sigma}){\bf 1}_{\{\tau >\sigma\}})] \right|\\
\nonumber &&\hspace{10mm} \leq 
\mathbb{E}_x[e^{-\alpha \tau }|\check{h}_N-h|(X_{\tau})]+
\mathbb{E}_x[e^{-\alpha \sigma }|\check{g}_N-g|(X_{\sigma})]\\
&& \hspace{10mm} \leq
\sup_{\tau} \mathbb{E}_x[e^{-\alpha \tau }|\check{h}_N-h|(X_{\tau})]+ \sup_{\sigma}
\mathbb{E}_x[e^{-\alpha \sigma }|\check{g}_N-g|(X_{\sigma})]. \label{eq:tilsad}
\end{eqnarray}
Since \(\check{H}_N:= \sup_{\tau} \mathbb{E}_x[e^{-\alpha \tau }|\check{h}_N-h|(X_{\tau})]\) is the solution of Theorem \ref{Zhang} replaced \(g\) by \(|\check{h}_N-h|\), it holds that
\[\cE_{\alpha}(\check{H}_N, \check{H}_N)\leq \cE_{\alpha}(\check{H}_N,|\check{h}_N-h|)\]
and, by the sector condition \((\ref{eq:sector})\), we have
\begin{eqnarray*}
\cE_{\alpha}(\check{H}_N, \check{H}_N)&\leq & K_{\alpha}^2 \cE_{\alpha}(|\check{h}_N-h|,|\check{h}_N-h|)\\
&\leq & K_{\alpha}^2 \cE_{\alpha}(\check{h}_N-h,\check{h}_N-h).
\end{eqnarray*}
Since \(\check{h}_N\) converges to \(h\) in \(\cE_{\alpha}\), \(\cE_{\alpha}(\check{H}_N, \check{H}_N)\) converges to \(0\) as \(N\to \infty.\) By \cite[Theorem 2.2.5]{O}, \(\check{h}_N\) converges to \(0\) quasi-uniformly, that is, there exist increasing sequence of closed set \(\{F_j\}_j\), subsequence \(\{N_k\}\) and the modification \(\check{H}_{N_k}\) such that \( \lim_{J\to \infty }\text{Cap}( E \setminus \bigcup_{j=1}^J F_j)=0\) and \(\check{H}_{N_k}\) converges to \(0\) uniformly on each \(F_j.\) Similarly, \(\check{G}_N\) converges to \(0\) quasi-uniformly.
By taking modifications of \(\check{v}_{N_k}\) and \(v\) if necessary, by \((\ref{eq:tilsad})\), it holds that \(|\check{v}_N-v|\leq \check{H}_N+\check{G}_N\). So \(\check{v}_{N}\) converges to \(v\) quasi-uniformly. The rest of the proof is similar to \cite[Section 4]{FM}, but we describe the proof for reader's convenience.

Since \(\check{v}_{N_k},\check{h}_{N_k}, \check{g}_{N_k}\) converges to \(v, h, g,\) quasi-uniformly respectively, there exists increasing sequence of closed sets \(\{F_j\}_j\) such that \( \lim_{J\to \infty }\text{Cap}( E \setminus \bigcup_{j=1}^J F_j)=0\),  \(\check{v}_{N_k},\check{h}_{N_k}, \check{g}_{N_k}\) converges to \(v, h, g,\) uniformly on each \(F_j.\) Set stopping times \(\tau_{\gamma}:=\inf\{t>0: v(X_t)+\gamma \geq h(X_t)\}\) and \(T_j:=\inf\{t>0: X_t\in E\setminus F_j \}\) for \(\gamma>0\) and  \(j\in \mathbb{N}.\) Without loss of generality, we may assume \(\mathbb{P}_x(\lim_{j\to \infty} T_j=\infty)=1\) q.e. \(x.\) It is enough to show that 
\begin{equation}
\mathbb{E}_x[e^{-\alpha \sigma}v(X_{\sigma})] \leq v(x)\leq \mathbb{E}_x[e^{-\alpha \tau}v(X_{\tau})]\label{eq:FMsaddle}
\end{equation}
for q.e. \(x\in E\) and stopping times \(\tau, \sigma\) satisfying \(\sigma \leq \hat{\tau}\) and \(\tau \leq \hat{\sigma}\) \(\mathbb{P}_x\)-a.e. Indeed, for any stopping time \(\sigma\), \((\ref{eq:FMsaddle})\) and \(g\leq v\) implies

\begin{eqnarray*}
J_x(\hat{\tau}, \sigma)&=&\mathbb{E}_x[e^{-\alpha \hat{\tau}}h(X_{\hat{\tau}}){\bf 1}_{\{\hat{\tau}\leq \sigma\}}+e^{-\alpha \sigma}g(X_{\sigma}){\bf 1}_{\{\hat{\tau} >\sigma\}}
]\\
&\leq &\mathbb{E}_x[e^{-\alpha (\hat{\tau}\wedge \sigma)}v(X_{\hat{\tau}\wedge \sigma})]\\
&\leq &v(x)
\end{eqnarray*}
for q.e. \(x.\) Similarly, for any stopping time \(\tau\), it holds that \(J_x(\tau, \hat{\sigma})\geq v(x)\) q.e. \(x.\)

Since the pair \((\check{h}_{N_k}, \check{g}_{N_k})\) satisfies the separability condition, by the same way as \cite{N}, \cite{FM} or the proof of Theorem \ref{mainThm} under the separability condition, we have
\begin{equation}
\mathbb{E}_x[e^{-\alpha \sigma}\check{v}_{N_k}(X_{\sigma})] \leq \check{v}_{N_k}(x)\leq \mathbb{E}_x[e^{-\alpha \tau}\check{v}_{N_k}(X_{\tau})]\label{eq:FMsaddle_N_k}
\end{equation}
for q.e. \(x\in E\) and stopping times \(\tau, \sigma\) satisfying \(\sigma \leq \hat{\tau}_{N_k}\) and \(\tau \leq \hat{\sigma}_{N_k}\) \(\mathbb{P}_x\)-a.e., where \(\hat{\tau}_{N_k}:=\inf\{t>0 : \check{v}_{N_k}(X_t)=\check{h}_{N_k}(X_t)\}\) and \(\hat{\sigma}_{N_k}:=\inf\{t>0 : \check{v}_{N_k}(X_t)=\check{g}_{N_k}(X_t)\}\).

Fix a stopping time \(\sigma\). For any \(\gamma>0,\) and \(j\in \mathbb{N}\), by the quasi-uniformly convergences, 
\[|\check{v}_{N_k}(x)-v(x)|\leq \frac{\gamma}{2},\ \  |\check{h}_{N_k}(x)-h(x)|\leq \frac{\gamma}{2}\]
for large \(k\) and \(x\in F_j\). Then it holds that \(\tau_{\gamma}\wedge T_j\wedge \sigma \leq \hat{\tau}_{N_k}\). Indeed, for any \(t<\tau_{\gamma}\wedge T_j\wedge \sigma\), it holds that
\[\check{v}_{N_k}(X_t)\leq v(X_t)+\frac{\gamma}{2}< h(X_t)-\frac{\gamma}{2}<\check{h}_{N_k}(X_t).\]
Thus \((\ref{eq:FMsaddle_N_k})\) implies that
\[\check{v}_{N_k}(x)\geq \mathbb{E}_x[e^{-\alpha(\tau_{\gamma}\wedge T_j\wedge \sigma)}\check{v}_{N_k}(X_{\tau_{\gamma}\wedge T_j\wedge \sigma})]\]
for any \(x\in F_j.\) Letting \(k, j\) tend to infinity and \(\gamma\) tend to \(0\), Lebesgue's convergence theorem, the uniform integrability of \(\{e^{-\alpha T}v(X_T)\}_{T }\) and the quasi-left continuity for \(X\) imply \(v(x)\geq \mathbb{E}_x[e^{-\alpha \sigma}v(X_{\sigma})]\). The another inequality in \((\ref{eq:FMsaddle})\) can be proved similarly, so the proof is completed.
\end{proof}

\section{Examples}\label{examples}
\subsection{Callable options}
Some callable options are examples of the Dynkin games. Considering a regular lower bounded semi-Dirichlet form \((\cE, \cF)\) on \(L^2(\mathbb{R};dx)\) and let \(X\) be the Hunt process associated with \((\cE, \cF)\) starting from \(x \in \mathbb{R}\). We treat \(X\) as the price process of an asset with discount rate \(\alpha\). Fix \(g, h \in \cF\) satisfying \(0\leq g\leq h\) a.e. We define a callable option as follows. If a buyer exercises the option at time \(\sigma\) before a seller redeems the option at \(\tau\), then the buyer gains \(e^{-\alpha \sigma}g(X_{\sigma})\). If the seller redeems the option at time \(\tau\) before the buyer exercises the option at \(\sigma\), then the buyer gains \(e^{-\alpha \tau}h(X_{\tau})\). The difference \(h-g\) is the penalty which the seller pays to the buyer. So the expectation of the gain of the buyer is \(J_x(\tau, \sigma)\). Because this is the zero-sum game, the seller tries to choose \(\tau\) so that \(J_x(\tau, \sigma)\) is less, and the buyer tries to choose \(\sigma\) so that \(J_x(\tau, \sigma)\) is more. Thus the equilibrium price is \(v(x)=\inf_{\tau} \sup_{\sigma} J_x(\tau, \sigma)=\sup_{\sigma} \inf_{\tau} J_x(\tau, \sigma)\) and \((\hat{\tau}, \hat{\sigma})\) defined in Section \ref{Introduction} attains \(v\).

Furthermore, consider the specific case of \((\cE,\cF)\). Let \(\lambda>0\) and \(\mu>0\). Assume that \(\cE\) is given by 
\begin{equation}
\cE(f,g):=\lambda \int_{\mathbb{R}} \frac{df}{dx} \frac{dg}{dx} dx - (\mu -\frac{\lambda^2}{2})\int_{\mathbb{R}}\frac{df}{dx} g dx
\end{equation}
for any \(f,g\in C_c^1(\mathbb{R})\). Then, by \cite[Theorem 1.5.2]{O}, there exists \(\alpha_0 \geq 0\) such that \((\cE.1)\) and \((\cE.2)\) hold for \(C_c^1(\mathbb{R})\).
Define \(\cF\) as the \(\cE_{\alpha_0}\)-closure of \(C_c^1(\mathbb{R})\), then \((\cE, \cF)\) is a regular lower bounded semi-Dirichlet form on \(L^2(\mathbb{R};dx)\) by \cite[Theorem 1.5.3]{O}. Let \(X\) be the Hunt process associated with \((\cE, \cF)\), then \(X_t\) has the same distribution as that of \(\lambda W_t + (\mu -\lambda^2/2)t\) for \(t\geq 0\), where \(W\) is Brownian motion starting from \(x\). Taking \(N>0\) large enough, for a constant \(K>0\) and a non-negative function \(p \in \cF\), we define \(g(x):=0 \vee (e^x-K){\bf 1}_{\{x\leq N\}}\) and \(h(x):=g(x)+p(x)\). 

The process \(\{e^{X_t}\}_t\) starting at \(e^{\lambda x}\) is called a geometric Brownian motion whose drift is \(\mu\) and volatility is \(\lambda\), and this is often used as a model of an asset. So, in the above settings, the solution of \((\ref{eq:varineq})\) is the equilibrium price of the callable put option with the strike price \(K\). 

\subsection{Diffusion process with drifts}
We apply Theorem \ref{mainThm} to the case of diffusion process treated in \cite[\S II.2. d]{MR}.

Let \(d\geq 3\), \(E:=U \subset \mathbb{R}^d\) be an open set and \(dm=dx\) be the Lebesgue measure on \(U\). For \(1\leq i, j\leq d\), let \(a_{ij} \in L^1_{loc}(U), b_i, d_i \in L^d_{loc}(U)\), and \(c \in L^{d/2}_{loc}(U)\) satisfying \(d_i -b_i \in L^d_{loc}(U) \cap L^{\infty}_{loc}(U)\),
\[\sum_{i=1}^d a_{ij}(x)\xi_i \xi_j \geq 0\ \text{\ for\ }\xi=(\xi_1, \cdots, \xi_d) \in \mathbb{R}^d\ \text{and\ a.e.\ }x,\]
and
\[\int (cu+\sum_{i=1}^d b_i  \frac{\partial u}{\partial x_i})dx \geq 0, \ \int (cu+\sum_{i=1}^d d_i  \frac{\partial u}{\partial x_i})dx \geq 0\]
for any \(u\in C_c^{\infty}(U)\) with \(u\geq 0.\) For any \(u,v \in C_c^{\infty}(U)\), we define
\[\cE(u,v):= \sum_{i,j=1}^d \int a_{ij}\frac{\partial u}{\partial x_i}\frac{\partial v}{\partial x_j}dx +\sum_{i=1}^d \int b_i\frac{\partial u}{\partial x_i}v dx+\sum_{j=1}^d \int d_{j}u\frac{\partial v}{\partial x_j}dx + \int cuvdx.\]

Assume that there exists \(c_1>0\) such that \[\sum_{i,j=1}^d \frac{a_{ij}(x)+a_{ji}(x)}{2}\xi_i \xi_j \geq c_1 \sum_{i=1}^d |\xi_i|^2\ \text{for\ any\ }\xi \in \mathbb{R}^d\text{\ and\ }m\text{-a.e.\ }x,\]
and 
\[\max_{i,j}\left|\frac{a_{ij}(x)-a_{ji}(x)}{2}\right| \text{\ is\ bounded\ for\ }m\text{-a.e.\ }x\]. Denote by \(\cF\) the \(\cE\)-closure of \(C_c^{\infty}(U)\), then \((\cE, \cF)\) is a regular non-symmetric Dirichlet form on \(L^2(U)\) by \cite[\S II.2. d]{MR}, where non-symmetric Dirichlet form is a lower bounded semi-Dirichlet form with \(\alpha_0=0\) and satisfying the following condition \((\hat{\cE}.4)\).\\
\((\hat{\cE}.4)\) For any \(u\in \cF\) and \(a\geq 0,\) it holds that \(u\wedge a \in \cF\) and \(\cE(u-u\wedge a, u\wedge a)\geq 0.\)

For \(g,h \in \cF\) satisfying \(g\leq h\) a.e., we consider the variational inequality as follows.
\begin{eqnarray*}
\cE_{\alpha}(v,u-v)&=&\sum_{i,j=1}^d \int a_{ij}\frac{\partial v}{\partial x_i}\frac{\partial(u-v)}{\partial x_j}dx +\sum_{i=1}^d \int b_i\frac{\partial v}{\partial x_i}(u-v) dx\\
&&+\sum_{j=1}^d \int d_{j}v\frac{\partial (u-v)}{\partial x_j}dx + \int (c+\alpha)v(u-v)dx \geq 0
\end{eqnarray*}
for any \(u\in \cF\) with \(g\leq u\leq h\), and \(g\leq v\leq h \). By Theorem \ref{mainThm}, \(v\) defined by \((\ref{eq:supinf})\) is the unique solution of this variational inequality.

\subsection{Jump process}
Next we apply Theorem \ref{mainThm} to jump processes which are special cases of semi-Dirichlet forms constructed in \cite{U}. Let \(d\geq 1\), \(b=(b_1, \cdots, b_d)\in \mathbb{R}^d\) and \(k(x,y)\) be a measurable function defined on \(\mathbb{R}^d \times \mathbb{R}^d \setminus \{(x,x); x\in \mathbb{R}^d\}.\) Assume the following conditions (J.1), (J.2) and (J.3').\\
(J.1) \(M_s \in L^1_{loc}(\mathbb{R}^d)\) for \(M_s(x):=\int_{\{y\neq x\}}(1\wedge |x-y|^2)k_s(x,y)dy\)\\
(J.2) There exists a constant \(\gamma \in (0,1]\) and \(C>0\) such that 
\[\sup_{x\in \mathbb{R}^d}\int_{|x-y|\geq 1}|k_a(x,y)|dy +\sup_{x\in \mathbb{R}^d}\int_{|x-y|< 1}|k_a(x,y)|^{\gamma}dy< \infty\]
and \(|k_a(x,y)|^{2-\gamma} \leq C k_s(x,y)\) for \(x,y \in \mathbb{R}^d\) with \(0<|x-y|<1.\) Here \(k_s\) and \(k_a\) are defined by
\[k_s(x,y):=\frac{k(x,y)+k(y,x)}{2},\ k_a(x,y):=\frac{k(x,y)-k(y,x)}{2}\]
for \(x,y \in \mathbb{R}^d\) with \(x\neq y,\) respectively.\\
(J.3') There exists a constant \(\kappa >0\) such that \(k(x,y)\geq \kappa |x-y|^{-d-1}\) for \(x,y \in \mathbb{R}^d\) with \(0<|x-y|<1\) and \(\sum_{i=1}^d |b_i| \leq \kappa \pi ^{d/2+1}/8\Gamma(d/2+1/2)\), where \(\Gamma\) is the Gamma function.

We define \((\cE,\cF)\) by
\begin{eqnarray*}
\cE(f,g)&:=&\frac{1}{2}\int\int_{\{x\neq y\}}(f(x)-f(y))(g(x)-g(y))k_s(x,y)dxdy \\
&& +\int\int_{\{x\neq y\}}(f(x)-f(y))g(x)k_a(x,y)dxdy + \sum_{i=1}^d b_i\int \frac{\partial f}{\partial x_i} \cdot g(x) dx
\end{eqnarray*}
for \(f,g \in C_c^1(\mathbb{R}^d)\), and \(\cF\) be the \(\cE\)-closure of \(C_c^1(\mathbb{R}^d)\). Then \((\cE,\cF)\) is a regular lower bounded semi-Dirichlet form on \(L^2(\mathbb{R}^d)\) by \cite{U}. By applying Theorem \ref{mainThm}, the unique solution of the variational inequality \((\ref{eq:varineq})\) is \(v\) defined in \((\ref{eq:supinf})\).

\end{document}